\pdfoutput=1
\RequirePackage{ifpdf}
\ifpdf 
\documentclass[pdftex]{sigma}
\else
\documentclass{sigma}
\fi

\numberwithin{equation}{section}

\newtheorem{Theorem}{Theorem}[section]
\newtheorem{Lemma}[Theorem]{Lemma}
\newtheorem{Proposition}[Theorem]{Proposition}
 { \theoremstyle{definition}
\newtheorem{Definition}[Theorem]{Definition}
\newtheorem{Example}[Theorem]{Example}
\newtheorem{Remark}[Theorem]{Remark} }

\newcommand{\smalltwobytwo}[4]{
\left( \begin{smallmatrix}
 #1 & #2\\
 #3 & #4
\end{smallmatrix}\right)}

\newcommand{\twobytwo}[4]{
\left( \begin{matrix}
 #1 & #2\\
 #3 & #4
\end{matrix} \right)}

\begin{document}

\allowdisplaybreaks

\newcommand{\arXivNumber}{1712.02437}

\renewcommand{\thefootnote}{}

\renewcommand{\PaperNumber}{071}

\FirstPageHeading

\ShortArticleName{The Chevalley--Weil Formula for Orbifold Curves}

\ArticleName{The Chevalley--Weil Formula for Orbifold Curves\footnote{This paper is a~contribution to the Special Issue on Modular Forms and String Theory in honor of Noriko Yui. The full collection is available at \href{http://www.emis.de/journals/SIGMA/modular-forms.html}{http://www.emis.de/journals/SIGMA/modular-forms.html}}}

\Author{Luca CANDELORI}

\AuthorNameForHeading{L.~Candelori}

\Address{Department of Mathematics, University of Hawaii at Manoa, Honolulu, HI, USA}
\Email{\href{mailto:candelori@math.hawaii.edu}{candelori@math.hawaii.edu}}
\URLaddress{\url{http://math.hawaii.edu/~candelori/}}

\ArticleDates{Received December 08, 2017, in final form July 02, 2018; Published online July 17, 2018}

\Abstract{In the 1930s Chevalley and Weil gave a formula for decomposing the canonical representation on the space of differential forms of the Galois group of a ramified Galois cover of Riemann surfaces. In this article we prove an analogous Chevalley--Weil formula for ramified Galois covers of orbifold curves. We then specialize the formula to the case when the base orbifold curve is the (reduced) modular orbifold. As an application of this latter formula we decompose the canonical representations of modular curves of full, prime level and of Fermat curves of arbitrary exponent.}


\Classification{14H30; 14H37; 14H45}

\renewcommand{\thefootnote}{\arabic{footnote}}
\setcounter{footnote}{0}

\section{Introduction}

Let $f\colon X\rightarrow Y$ be a ramified Galois cover of compact Riemann surfaces of genus $g_X$ and $g_Y$, and let $G := \mathrm{Aut}_X(Y)$ be the Galois group of $f$. This is a finite group of automorphisms of $X$ that acts by pull-back on the space of holomorphic differential 1-forms of $X$, giving a $g_X$-dimensional complex representation
\begin{gather*}
\rho_f\colon \ G^{\rm{op}}\rightarrow \mathrm{GL}\big(H^0\big(X,\Omega^1_X\big)\big),
\end{gather*}
the {\em canonical} representation of $G$. Given an irreducible representation $\rho$ of $G$, Chevalley and Weil \cite{Chevalley--Weil,Weil} gave a formula for the multiplicity of~$\rho$ inside~$\rho_f$ in terms of the genus of~$Y$ and the ramification data of $f$. This is called the {\em Chevalley--Weil formula}. A modern account of this result in the more general setting of algebraic curves over an algebraically closed field whose characteristic does not divide $|G|$ can be found in~\cite{Naeff}. The formula has subsequently been generalized by Nakajima \cite{Nakajima} to any coherent sheaf and any ramified cover of algebraic varieties over an any algebraically closed field (the generalization of the Chevalley--Weil formula for curves over any algebraically closed field was also treated in \cite{Kani}). The Chevalley--Weil formula can be used in conjunction with the character table of $G$ to write down explicit matrices for $\rho_f$ without having to compute a basis for $H^0\big(X,\Omega^1_X\big)$. As an application, when $g_X\geq 3$ is small explicit equations for the canonical embedding $X\rightarrow \mathbb{P}^{g_X-1}$ can be found using this model for $\rho_f$~\cite{Streit}.

In this article we generalize the Chevalley--Weil formula to ramified Galois covers of {\em orbifold curves} (also known as {\em Deligne--Mumford curves} or {\em stacky curves}), under the mild assumption that the ramification locus is disjoint from the locus of orbifold points in~$Y$. We work over the complex numbers, but the same exact arguments go through unchanged for an arbitrary algebraically closed field whose characteristic does not divide $|G|$. The proof is contained in Section~\ref{section:CWFormula}. The structure of the proof follows the excellent exposition in \cite{Naeff}, with the necessary generalizations and extra computations needed in the orbifold case. We have also simplified the arguments when possible by referring to the literature.

Suppose now that the base orbifold curve $Y=X(1)$ is the compactification of the orbifold quotient $\mathfrak{h}/\mathrm{PSL}_2(\mathbb{Z})$, the reduced modular orbifold. The Galois covers of $X(1)$ are modular curves $X(\Gamma)$ corresponding to normal subgroups $\Gamma\triangleleft\mathrm{PSL}_2(\mathbb{Z})$. The canonical representations $\rho_{\Gamma}$ in this case can be viewed as representations of $\mathrm{PSL}_2(\mathbb{Z})^{\rm{op}}$ factoring through the finite quotient $G:= \mathrm{PSL}_2(\mathbb{Z})/\Gamma$. Our orbifold Chevalley--Weil formula in this case applies and we are able to compute explicit matrices for $\rho_{\Gamma}$ whenever all its factors are of dimension $\leq 6$, by consulting the wealth of information available on the representation theory of $\mathrm{PSL}_2(\mathbb{Z})$ (e.g., \cite{LeBruyn, Mason2,TubaWenzl}). We do so in Section~\ref{section:modularCurves}. In Section~\ref{section:principalCongruenceSubgroups} we specialize to the case of principal congruence subgroups $\Gamma = \Gamma(p)$, for $p \geq 5$ a prime. In this case $G \simeq \mathrm{PSL}_2(\mathbb{F}_p)$ and the character tables of these groups have been known since Frobenius and Schur. For each irreducible representation $\rho$ appearing in the character table of $\mathrm{PSL}_2(\mathbb{F}_p)$, we find a formula for the multiplicity of $\rho$ inside the canonical representation $\rho_{\Gamma(p)}$. We then apply our computations to compute the canonical representation for the Klein quartic, which can be uniformized by the principal congruence subgroup~$\Gamma(7)$. In Section~\ref{section:FermatCurves}, we present a uniformization of Fermat curves $F_N$ of exponent $N$ by normal subgroups $\Phi(N)\triangleleft \mathrm{PSL}_2(\mathbb{Z})$, which exhibits $F_N$ as a Galois cover of $X(1)$ with Galois group isomorphic to $(\mathbb{Z}/N\mathbb{Z})^2\rtimes S_3$. This is known to be the full automorphism group of $F_N$ \cite{Tzermias}, and is thus interesting to compute the decomposition of the canonical representation $\rho_{\Phi(N)}$. We do so using our Chevalley--Weil formula in Theorems~\ref{thm:CWFermatCurveNDoesNotDivide3} and~\ref{thm:CWFermatCurve3DividesN} below. As shown in~\cite{Barraza-Rojas}, the decomposition of $\rho_{\Phi(N)}$ into irreducible representations gives a corresponding `group algebra' decomposition of the Jacobian variety of $F_N$. Decomposing the Jacobian variety of $F_N$ is an important problem that has been previosuly considered by many in algebraic geometry and number theory, including Noriko Yui \cite{Yui}.

\section{Vector bundles over orbifold curves}

In this section we recall a few basic facts about vector bundles on orbifold curves. Good references for these facts are \cite{FurutaSteer} and \cite{Nasatyr-Steer}. Alternatively, the reader may consult the literature on locally free sheaves over {\em Deligne--Mumford curves} or {\em stacky curves} \cite{behrend-noohi}. In this article an {\em orbifold curve} $X$ is a compact, connected, complex orbifold of dimension one with finitely many orbifold points $(P_1, \ldots, P_n)$ with non-trivial cyclic stabilizers of orders $(p_1, \ldots, p_n)$, respectively. We assume that $X$ is generically a Riemann surface (thus a {\em reduced} orbifold). The {\em genus} of~$X$, denoted by~$g_X$, is the genus of its underlying Riemann surface $\bar{X}$ (this coincides with the {\em coarse moduli space} if~$X$ is viewed as a Deligne--Mumford curve). Let $\alpha_1,\ldots, \alpha_{g_X},\beta_1,\ldots,\beta_{g_X}$ be a set of standard generators for the fundamental group of $\bar{X}$. The {\em orbifold fundamental group} of $X$ (with respect to some base-point $b$) can be defined as
\begin{gather}
\pi_1(X;b):= \big\{ \alpha_1,\ldots, \alpha_{g_X},\beta_1,\ldots,\beta_{g_X},\gamma_{P_1},\ldots, \gamma_{P_n}\colon \nonumber\\
\hphantom{\pi_1(X;b):= \big\{}{} \gamma_{P_1}^{p_1}=1,\ldots,\gamma_{P_n}^{p_n}=1, \gamma_{P_1}\cdots\gamma_{P_n}[\alpha_1,\beta_1]\cdots[\alpha_{g_X},\beta_{g_X}] =1 \big\}, \label{eq:fundamentalGroup}
\end{gather}
where $\gamma_{P_i}$ is an oriented generator of the stabilizer of the orbifold point $P_i$, $i=1,\ldots, n$. Let $\mathrm{Pic}(X)$ be the group under $\otimes$ of all line bundles (i.e invertible sheaves) over $X$ up to isomorphism, and let
\begin{gather*}
\deg\colon \ \mathrm{Pic}(X) \longrightarrow \frac{1}{m}\mathbb{Z}, \qquad m = \mathrm{lcm}(p_1,\ldots,p_n)
\end{gather*}
be the usual `degree' homomorphism, obtained by summing divisor multiplicities, possibly rational numbers at the orbifold points \cite[Section~1B]{Nasatyr-Steer}. For example, the sheaf $\Omega^1_X$ of holomorphic differential 1-forms on $X$ satisfies \cite[Section~1]{FurutaSteer}, \cite[Section~1A]{Nasatyr-Steer}
\begin{gather}\label{equation:canonicalDegree}
\deg \Omega^1_X = 2g_X - 2 + \sum_{i=1}^n \frac{p_i-1}{p_i} = 2g_X - 2 + n - \sum_{i=1}^n \frac{1}{p_i}.
\end{gather}

Let $\mathcal{V}$ be a vector bundle over $X$ (i.e., a locally free sheaf of finite rank). The {\em degree} of a~vector bundle is the rational number $\deg \det \mathcal{V}$. The fiber of $\mathcal{V}$ at each orbifold point $P_i$ gives a~$\mathrm{rk}(\mathcal{V})$-dimensional representation
\begin{gather*}
\mu(\mathcal{V},P_i)\colon \ \langle \gamma_{P_i} \rangle \longrightarrow \mathrm{GL}_{\mathrm{rk}(\mathcal{V})}(\mathbb{C}),
\end{gather*}
which is entirely determined by the linear transformation $\mu(\mathcal{V},P_i)(\gamma_{P_i})$. This is of finite order, hence diagonalizable, with eigenvalues of the form $e^{\frac{2\pi i}{p_i} \nu_{ij}}$, $\nu_{ij} \in \{0, \ldots, p_i-1\}$, $j=1,\ldots, \mathrm{rk}(\mathcal{V})$.

\begin{Definition}[\cite{FurutaSteer}]\label{definition:isotropyTrace}
The integers $\nu_{ij} \in \{0, \ldots, p_i-1\}$, $j=1, \ldots, \mathrm{rk}(\mathcal{V})$ are called the {\em isotropies} of $\mathcal{V}$ at $P_i$. The integer
\begin{gather*}
\iota(\mathcal{V},P_i):= \sum_{i=1}^{\mathrm{rk}(\mathcal{V})} \nu_{ij} \in \mathbb{Z}_{\geq 0}
\end{gather*}
is called the {\em isotropy trace} of $\mathcal{V}$ at $P_i$.
\end{Definition}

The {\em cohomology groups} $H^i(X,\mathcal{V})$ of a vector bundle $\mathcal{V}$ over $X$ are defined as the sheaf cohomology over the site defined by $X$. Denote by
\begin{gather*}
\chi(X,\mathcal{V}) := \dim H^0(X,\mathcal{V}) - \dim H^1(X,\mathcal{V})
\end{gather*}
the Euler characteristic of the vector bundle $\mathcal{V}$ over $X$. We have:

\begin{Theorem}[Riemann--Roch theorem for orbifold curves]\label{theorem:RRochOrbifold}
Let $\mathcal{V}$ be a vector bundle over an orbifold curve $X$. Then
\begin{gather*}
\chi(X,\mathcal{V}) = \mathrm{rk}(\mathcal{V})(1 - g_X) + \deg(\mathcal{V}) - \left( \sum_{i=1}^n \frac{\iota(\mathcal{V},P_i)}{p_i} \right).
\end{gather*}
\end{Theorem}

\begin{proof}See, e.g., \cite[Theorem~1.5]{FurutaSteer} for the case of line bundles. The general case follows from the splitting principle.
\end{proof}

In addition, the Serre duality theorem for orbifold curves \cite[Theorem~1.7]{Nasatyr-Steer} gives a canonical isomorphism
\begin{gather}\label{equation:SerreDuality}
H^1(X, \mathcal{V}) \simeq H^0\big(X, \mathcal{V}^*\otimes\Omega^1_X\big)^*,
\end{gather}
where $\mathcal{V}^*:= \operatorname{Hom}_{\mathcal{O}_X}(\mathcal{V},\mathcal{O}_X)$.

\section{The Chevalley--Weil formula}\label{section:CWFormula}

Let $f\colon X\rightarrow Y$ be a degree $d$ ramified Galois cover of orbifold curves of genus $g_X$, $g_Y$, respectively. By this we mean that $f$ is a morphism that is generically a degree $d$ finite \'{e}tale Galois cover, outside finitely many {\em ramification points} $Q_1, \ldots, Q_r \in Y$, with ramification degrees $e_1, \ldots, e_r$, respectively. We assume for simplicity that none of the $Q_i$'s are orbifold points of $Y$.

Let
 \begin{gather*}
 G:=\mathrm{Aut}_Y(X)
 \end{gather*}
be the group of covering transformations of $f$. For each ramification point $Q\in Y$, the stabilizer in $G$ of any point $ R\stackrel{f}\longmapsto Q$ of $X$ is a cyclic subgroup $G_{R} \subseteq G$ of order $e$, the ramification degree of $Q$. Choose generators $\gamma_{R} \in G_{R}$ for all such `local monodromy' groups. Note that if $R,R' \stackrel{f}\longmapsto Q$ are points of $X$ lying over $Q$, then $G_{R}$ and $G_{R'}$ are conjugate to each other and there exists $g\in G$ with
\begin{gather} \label{eq:conjugateLocalMonodoromy}
g^{-1} \gamma_R g = \gamma_{R'}.
\end{gather}
For each orbifold point $P_j \in Y$, $j=1, \ldots, n$, let $\gamma_{P_j}\in \pi_1(Y\backslash \{Q_1,\ldots,Q_r\};b)$ be an oriented generator for the stabilizer of the point $P_j$, as in \eqref{eq:fundamentalGroup}. We will abuse notation and also denote by $\gamma_{P_j}$ the image via the quotient map $\pi_1(Y\backslash \{Q_1,\ldots,Q_r\};b)\rightarrow G$.

The sheaf $f_*\mathcal{O}_X$ is a vector bundle over $Y$ of rank $d$, together with a right action of $G$ by pull-back of functions. This vector bundle is generically a finite \'{e}tale sheaf with right $G$-action, which on the stalks is just the right regular representation of $G$ on itself. The group $G$ also acts linearly on the line bundle $\Omega^1_X$ (and on $f_*\Omega^1_X$) by pull-back
\begin{gather*}
g\omega := g^*\omega,
\end{gather*}
giving the {\em canonical} representation
\begin{gather}
\label{eq:canonicalRep}
\rho_f\colon \ G^{\rm{op}} \longrightarrow \mathrm{GL}\big(H^0\big(X,\Omega^1_X\big)\big) \simeq \mathrm{GL}_{g_X}(\mathbb{C}),
\end{gather}
by taking global sections. For ease of notation we will drop the `op' super-script in what follows. Since $G$ is finite, the category of finite-dimensional representations of $G$ is semi-simple. Let~$\{\rho_i\}$ be the set of irreducible representations of $G$, and let $d_i = d(\rho_i)$ be the multiplicity~$d_i$ of each~$\rho_i$ occurring in~$\rho_f$. The {\em Chevalley--Weil formula} for algebraic curves \cite{Chevalley--Weil,Weil} is a formula for each $d_i$ in terms of global invariants of $f\colon X\rightarrow Y$ and ramification data around each local monodromy~$\gamma_R$. To generalize the formula to orbifold curves, we first set some notation. For $g\in G$ an element of exact order $N$, a representation $\rho$ of $G$, and $k = 0,\ldots, N-1$, let
\begin{gather*}
N_k(\rho; g) := \dim \ker \big( \rho(g) - e^{\frac{2\pi i}{N} k} I\big).
\end{gather*}
Our goal is to prove the following:

\begin{Theorem}[Chevalley--Weil formula]\label{theorem:Chevalley--Weil}
Let $\rho_i$ be an irreducible representation of~$G$. For each ramification point $Q_j \in Y$ choose some $R_j\stackrel{f}\mapsto Q_j$. Then the multiplicity $d_i = d(\rho_i)$ of $\rho_i$ in $\rho_f$ is given by
\begin{gather*}
d_i = \epsilon + \dim \rho_i\left(g_Y-1 + n - \sum_{j=1}^n \frac{1}{p_j} \right) + \sum_{j=1}^r \sum_{k=1}^{e_j-1} N_k(\rho_i;\gamma_{R_j})\left(1 - \frac{k}{e_j}\right) \\
\hphantom{d_i =}{} - \sum_{j=1}^n \sum_{k=1}^{p_j-1} N_k\Big(e^{\frac{-2\pi i}{p_j}}\rho_i;\gamma_{P_j}\Big)\frac{k}{p_j},
\end{gather*}
where $\epsilon = 1$ if $\rho_i = 1$ and $\epsilon = 0$ otherwise.
\end{Theorem}

\begin{Remark}Note that by \eqref{eq:conjugateLocalMonodoromy} the integers $N_{\rho_i}(\gamma_{R_j}; k)$ do not depend on the choice of point $R_j\stackrel{f}\mapsto Q_j$.
\end{Remark}

To prove Theorem~\ref{theorem:Chevalley--Weil} we proceed by a series of reductions, as in the case of algebraic cur\-ves~\cite{Naeff}. First, since $f$ is generically \'{e}tale we have an exact sequence
\begin{gather*}
0 \rightarrow f^*\Omega^1_Y \rightarrow \Omega^1_X \rightarrow \Omega^1_{X/Y} \rightarrow 0.
\end{gather*}
Applying $f_*$ to this exact sequence and the projection formula on the first term we get
\begin{gather}\label{equation:exactSeqf*}
0 \rightarrow \Omega^1_Y\otimes f_*\mathcal{O}_{X} \rightarrow f_*\Omega^1_X \rightarrow f_*\Omega^1_{X/Y} \rightarrow 0.
\end{gather}
This is an exact sequence of $\mathcal{O}_Y$-modules with right $\mathcal{O}_Y$-linear $G$-action. For any irreducible representation $\rho_i$ of $G$ we can take the $\rho_i$-isotypical component of this exact sequence and obtain (note that the action of $G$ on $Y$ is trivial):
\begin{gather*}
0 \rightarrow \Omega^1_Y\otimes f_*\mathcal{O}_{X}^{\rho_i} \rightarrow \big(f_*\Omega^1_X\big)^{\rho_i} \rightarrow \big(f_*\Omega_{_{X/Y}}^1\big)^{\rho_i} \rightarrow 0.
\end{gather*}
To prove Theorem~\ref{theorem:Chevalley--Weil} we have to compute
\begin{gather*}
\dim H^0\big(Y, \big(f_*\Omega^1_X\big)^{\rho_i}\big) = \dim H^0\big(X, \Omega^1_X\big)^{\rho_i} = d_i \cdot \dim \rho_i.
\end{gather*}
To do so we compute the Euler characteristic:
\begin{gather*}
\chi\big(Y, \big(f_*\Omega^1_X\big)^{\rho_i}\big) = \dim H^0\big(X, \Omega^1_X\big)^{\rho_i} - \dim H^1\big(X, \Omega^1_X\big)^{\rho_i}.
\end{gather*}
\begin{Proposition}\label{prop:GActionOnH1Trivial} The action of $G$ on $H^1\big(X, \Omega^1_X\big)$ is trivial. Therefore
\begin{gather*}
\dim H^1\big(X, \Omega^1_X\big)^{\rho_i} = \epsilon = \begin{cases} 1& \text{if } \rho_i = 1, \\
0& \text{if } \rho_i \neq 1,
 \end{cases}
\end{gather*}
for any irreducible representation $\rho_i$ of $G$.
\end{Proposition}

\begin{proof}By Serre duality \eqref{equation:SerreDuality} there is a canonical isomorphism $H^1\big(X,\Omega^1_X\big) \simeq H^0(X,\mathcal{O}_X)^*$, commuting with the linear $G$-action on both sides. Since $H^0(X,\mathcal{O}_X)\simeq \mathbb{C}$ consists of constant functions, the action of $G$ on this vector space is trivial. Therefore its contragradient action on $H^1\big(X,\Omega^1_X\big)$ is trivial as well.
\end{proof}

By Proposition~\ref{prop:GActionOnH1Trivial}, to compute $d_i$ it suffices to compute $\chi\big(Y, \big(f_*\Omega^1_X\big)^{\rho_i}\big)$. Since Euler cha\-racteristics are additive with respect to exact sequences, we must have
\begin{gather*}
\chi\big(Y, \big(f_*\Omega^1_X\big)^{\rho_i}\big) = \chi\big(Y, \big(f_*\Omega_{_{X/Y}}^1\big)^{\rho_i} \big) + \chi\big(Y,\Omega^1_Y\otimes f_*\mathcal{O}_{X}^{\rho_i}\big),
\end{gather*}
by applying $\chi(Y,-)$ to \eqref{equation:exactSeqf*}. We compute each summand in Propositions~\ref{prop:ramificationPiece} through~\ref{prop:isotropyV} below.

\begin{Proposition}\label{prop:ramificationPiece} Let $\rho_i$ be an irreducible representation of $G$. For each ramification point $Q_j \in Y$ choose some $R_j\stackrel{f}\rightarrow Q_j$. Then
\begin{gather*}
\chi\big(Y, \big(f_*\Omega_{X/Y}^1\big)^{\rho_i} \big) = \dim \rho_i \sum_{j=1}^r \mathrm{rk} ( \rho_i(\gamma_{R_j}) - I ).
\end{gather*}
\end{Proposition}

\begin{proof} The proof given in \cite{Naeff} for algebraic curves goes through unchanged, since the question is localized at the ramification locus (and disjoint from the orbifold locus).
\end{proof}

It remains to compute the Euler characteristic $\chi\big(Y,\Omega^1_Y\otimes f_*\mathcal{O}_{X}^{\rho_i}\big)$. Let $\mathcal{V} = \Omega^1_Y\otimes f_*\mathcal{O}_{X}^{\rho_i}$. By Theorem~\ref{theorem:RRochOrbifold},
\begin{gather*}
 \chi(Y,\mathcal{V}) = \mathrm{rk}(\mathcal{V})(1-g_Y) + \deg \mathcal{V} - \left( \sum_{i=1}^n \frac{\iota(\mathcal{V},P_i)}{p_i} \right).
\end{gather*}
We compute each term in turn.
\begin{Proposition}\label{prop:rankV}
\begin{gather*}
\mathrm{rk} \mathcal{V} = (\dim \rho_i)^2.
\end{gather*}
\end{Proposition}

\begin{proof}The vector bundle $f_*\mathcal{O}_{X}$ is a finite \'{e}tale sheaf of rank $d$ away from the ramification points, and the same holds for $f_*\mathcal{O}_{X}^{\rho_i}$. At a point $P\in Y$, $P\notin \{Q_1, \ldots, Q_r\}$, the fiber of $f_*\mathcal{O}_{X}$ is the right regular $G$-representation. As is well-known, the representation $\rho_i$ occurs inside the regular representation $\dim \rho_i$ times. Therefore the fiber at $P$ of $f_*\mathcal{O}_{X}^{\rho_i}$ has dimension $(\dim \rho_i)^2$ and that is also the rank of $f_*\mathcal{O}_{X}^{\rho_i}$ as a vector bundle over~$Y$. Since $\Omega^1_Y$ has rank one, the rank of~$\mathcal{V}$ is also of rank $(\dim \rho_i)^2$.
\end{proof}

To compute the degree of $\mathcal{V}$, note that
\begin{gather*}
\deg(\mathcal{V}) = (\dim \rho_i)^2 \deg \big(\Omega^1_Y\big) + \deg \big(f_*\mathcal{O}_{X}^{\rho_i}\big)\\
\hphantom{\deg(\mathcal{V})}{} = (\dim \rho_i)^2\left(2g_Y - 2 + n - \sum_{j=1}^n \frac{1}{p_j} \right) + \deg \big(f_*\mathcal{O}_{X}^{\rho_i}\big)
\end{gather*}
by \eqref{equation:canonicalDegree}. It remains to compute $\deg (f_*\mathcal{O}_{X}^{\rho_i})$, as follows.

\begin{Proposition}\label{prop:parabolicPiece}
\begin{gather*}
\deg \big(f_*\mathcal{O}_{X}^{\rho_i}\big) = -\dim \rho_i \sum_{j=1}^r \sum_{k=1}^{e_j-1} N_k(\rho_i;\gamma_{R_j})\frac{k}{e_j}.
\end{gather*}
\end{Proposition}

\begin{proof}The vector bundle $f_*\mathcal{O}_{X}^{\rho_i}$ is a direct sum of $\dim \rho_i$-copies of the {\em parabolic bundle} canoni\-cal\-ly associated to the unitary representation $\rho_i$, viewed as a representation of $
\pi_1(Y\backslash \{Q_1,\ldots$, $Q_r\};b)$ factoring through the finite quotient~$G$~\cite{Mehta-Seshadri}. The parabolic structure is supported at the ramification points $Q_1, \ldots, Q_r$ and for each $j$ it is determined by the local monodromy~$\gamma_{R_j}$, as follows. Since $\rho_i(R_j)$ is a matrix of order $e_j$, its eigenvalues are of the form $e^{2\pi i k_w/e_j}$, $k_w\in \{0,\ldots, e_j-1\}$. The parabolic structure at $Q_j$ is given by the decreasing filtration of eigenspaces of~$\gamma_{R_j}$, the weights are the rational numbers~$k_w/e_j$, $0\leq k_w/e_j <1$, and each weight~$k_w$ has multiplicity $\dim \rho_i N_{k_w}(\rho_i;\gamma_{R_j})$. The formula for the degree of $f_*\mathcal{O}_{X}^{\rho_i}$ then just follows from the well-known formula for the degree of a parabolic bundle \cite[Corollary~1.10]{Mehta-Seshadri}.
\end{proof}

Finally we compute the isotropy term in $\chi(Y,\mathcal{V})$:

\begin{Proposition}\label{prop:isotropyV}
\begin{gather*}
\left( \sum_{i=1}^n \frac{\iota(\mathcal{V},P_i)}{p_i} \right)=\sum_{j=1}^n \sum_{k=1}^{p_j-1} N_k\big(e^{\frac{-2\pi i}{p_j}}\rho_i;\gamma_{P_j}\big)\frac{k}{p_j}.
\end{gather*}
\end{Proposition}

\begin{proof}Note that the action of $\gamma_{P_j}$ on the fiber of $\mathcal{V}=\Omega^1_Y\otimes f_*\mathcal{O}_{X}^{\rho_i}$ over the orbifold point $P_j$ is given by $e^{-2\pi i/p_j}\rho_i$. The identity then follows from Definition~\ref{definition:isotropyTrace}.
\end{proof}

Theorem \ref{theorem:Chevalley--Weil} now follows by putting together Propositions~\ref{prop:GActionOnH1Trivial}, \ref{prop:ramificationPiece}, \ref{prop:rankV}, \ref{prop:parabolicPiece} and \ref{prop:isotropyV}.

\section{Applications to modular curves}\label{section:modularCurves}

We now apply Theorem~\ref{theorem:Chevalley--Weil} to the special case when the base orbifold $Y$ is the compactifica\-tion~$X(1)$ of the orbifold quotient $\mathrm{PSL}_2(\mathbb{Z})\backslash\mathfrak{h}$ obtained by adding the cusp~$\infty$. The genus of this orbifold curve is zero, $\pi_1(X(1)-\infty;b) \simeq \mathrm{PSL}_2(\mathbb{Z})$ and $X(1)$ has $n = 2$ orbifold points $P_1 = [i]$, $P_2 = [e^{2\pi i /3}]$ of orders $p_1 = 2$, $p_2 = 3$, with oriented stabilizers generated by
\begin{gather*}
\gamma_{P_1} := S = \twobytwo{0}{-1}{1}{0}, \qquad \gamma_{P_2} := R^{-1} = \twobytwo{1}{1}{-1}{0}.
\end{gather*}
Any finite-index normal subgroup $\Gamma \triangleleft \mathrm{PSL}_2(\mathbb{Z})$ gives a Galois cover
\begin{gather*}
f(\Gamma)\colon \ X(\Gamma) \rightarrow X(1),
\end{gather*}
where $X(\Gamma)$ is the compactification of the quotient $\Gamma\backslash\mathfrak{h}$ obtained by adding finitely many cusps. The map $f(\Gamma)$ is of degree $d=[\Gamma: \mathrm{PSL}_2(\mathbb{Z})]$, and it is ramified only above the cusp $\infty$. The Galois group of the cover is therefore $ G = \mathrm{PSL}_2(\mathbb{Z})/\Gamma$. To compute the ramification data, we may choose the cusp $\infty$ of $X(\Gamma)$, which lies above the cusp $\infty$ of $X(1)$. The local monodromy at $\infty$ is the image of
\begin{gather*}
\gamma_{\infty} := T = \twobytwo{1}{1}{0}{1}
\end{gather*}
under the quotient map $\mathrm{PSL}_2(\mathbb{Z})\rightarrow G$. The ramification degree $e$ of $f(\Gamma)$ over $\infty$ is just the order of the image of $T$ in $G$ (the {\em ramification level} of $G\triangleleft \mathrm{PSL}_2(\mathbb{Z})$). There is a canonical isomorphism
\begin{align*}
S_2(\Gamma) &\longrightarrow H^0\big(X(\Gamma), \Omega^1_{X(\Gamma)}\big), \\
f &\longmapsto f d\tau,
\end{align*}
so that the canonical representation $\rho_{\Gamma}:= \rho_{f(\Gamma)}$ \eqref{eq:canonicalRep} goes over to the representation $\mathrm{PSL}_2(\mathbb{Z})\rightarrow \mathrm{GL}(S_2(\Gamma))$ given by
\begin{gather*}
\rho_{\Gamma}(\gamma) f = f|_2 \gamma = f\left( \frac{ a\tau + b}{c\tau + d}\right)(c\tau + d)^{-2}, \qquad \gamma = \twobytwo{a}{b}{c}{d}\in \mathrm{PSL}_2(\mathbb{Z}),
\end{gather*}
factoring through the finite group $G$. Note that if $\rho_{\Gamma}(\gamma) f = f$ for all $\gamma \in \mathrm{PSL}_2(\mathbb{Z})$ then $f\in S_2(\mathrm{PSL}_2(\mathbb{Z})) = 0$, therefore the trivial representation can never occur inside $\rho_{\Gamma}$. For the remaining non-trivial representations of $G$ Theorem~\ref{theorem:Chevalley--Weil} simplifies to

\begin{Theorem}[Chevalley--Weil formula for $\Gamma\triangleleft \mathrm{PSL}_2(\mathbb{Z})$]\label{theorem:Chevalley--WeilModForms}
Let $\Gamma\triangleleft \mathrm{PSL}_2(\mathbb{Z})$ be a finite-index normal subgroup, and let $G:=\mathrm{PSL}_2(\mathbb{Z})/\Gamma$. Let $\rho\neq 1$ be an irreducible representation of $G$. Then the multiplicity of $\rho$ in $\rho_{\Gamma}$ is given by
\begin{gather*}
d = -\frac{5}{12}\dim \rho + \sum_{k=1}^{e-1} N_k(\rho;T)\left(1 - \frac{k}{e}\right) - \frac{\operatorname{Tr}(\rho(S))}{4} + \frac{\zeta^2_3\operatorname{Tr}\big(\rho\big(R^{-1}\big)\big)}{3(1-\zeta_3^{2})} + \frac{\zeta_3\operatorname{Tr}(\rho(R))}{3(1-\zeta_3)},\end{gather*}
where $\zeta_3:= e^{2\pi i/3}$.
\end{Theorem}

\begin{proof}Using the notation of Theorem~\ref{theorem:Chevalley--Weil}, note first that $\epsilon = 0$ always since we are assuming $\rho \neq 1$. We also have that $g_{X(1)}=0$, $n=2$, $p_1 = 2$, $p_2 = 3$, which gives a contribution of $\dim \rho/6$ in Theorem~\ref{theorem:Chevalley--WeilModForms}. As stated above, there is only one ramification point $\infty \in X(1)$ of index $e$ and with local monodromy given by the image of $T$ in $G$, which gives the second term. As for the remaining terms, since the order of $\rho(S)$ and $\rho(R^{-1})$ is low it is more convenient computationally to express the eigenvalue multiplicities in terms of traces. In particular, note that
\begin{gather*}
\operatorname{Tr}(\rho(S)) = N_0(\rho(S)) - N_1(\rho(S)) = N_1(-\rho(S)) - N_0(-\rho(S)) = 2N_1(-\rho(S)) - \dim \rho
\end{gather*}
and similarly
\begin{gather*}
\frac{N_1\big(e^{-2\pi i /3}\rho;R^{-1}\big)}{3} + 2\cdot\frac{N_2\big(e^{-2\pi i /3}\rho;R^{-1}\big)}{3} = \frac{\dim \rho}{3}-\frac{\zeta^2_3\operatorname{Tr}\big(\rho\big(R^{-1}\big)\big)}{3\big(1-\zeta_3^{2}\big)} - \frac{\zeta_3\operatorname{Tr}(\rho(R))}{3(1-\zeta_3)},
\end{gather*}
yielding the simplified formula.
\end{proof}

The formula of Theorem~\ref{theorem:Chevalley--WeilModForms} severely restricts which irreducible representations of $\mathrm{PSL}_2(\mathbb{Z})$ may occur inside a canonical representation. For example of the six characters of $\mathrm{PSL}_2(\mathbb{Z})$ only one may occur, namely the character $\chi$ defined by $\chi(T) = e^{2\pi i /6}$. This character must therefore be the canonical representation $\rho_{\Gamma}$ corresponding to {\em any} genus one normal subgroup $\Gamma\triangleleft\mathrm{PSL}_2(\mathbb{Z})$ (there are infinitely many of them).

Of the 27 two-dimensional irreducible representations of $\mathrm{PSL}_2(\mathbb{Z})$ which factor through a finite group (these have been classified in \cite{Mason2}) only four may appear in a canonical representation, namely the ones with $\rho(T)$ equivalent to
\begin{gather}\label{eqn:dim2IrrCanonical}
\twobytwo{\zeta_{12}}{0}{0}{\zeta_{12}^5}, \qquad \twobytwo{\zeta_8}{0}{0}{\zeta_8^3}, \qquad
\twobytwo{\zeta_{20}}{0}{0}{\zeta^9_{20}}, \qquad
\twobytwo{\zeta^3_{20}}{0}{0}{\zeta^7_{20}},
\end{gather}
where $\zeta_n:= e^{2\pi i /n}$. For any one of the four representations $\rho$ above, we can use GAP and a list of normal subgroups of $\mathrm{PSL}_2(\mathbb{Z})$ to identify the normal subgroup $\Gamma$ of smallest index such that $\rho$ is a factor of $\rho_{\Gamma}$. The first representation appears as a factor inside the canonical representation of the unique genus three subgroup of index 48, the second one is the canonical representation corresponding to the unique genus two subgroup (also of index~48), the third one appears as a~factor inside the canonical representation of the unique genus 15 subgroup of index 240 and the last representation appears as a~factor inside the canonical representation of the unique genus~55 subgroup of index~720.

A similar analysis could in principle be applied to higher-dimensional irreducible representations of $\mathrm{PSL}_2(\mathbb{Z})$ factoring through a finite group. If a given irreducible representation $\rho$ satisfies $d(\rho)>0$ in Theorem~\ref{theorem:Chevalley--WeilModForms}, what is the minimal index of the normal subgroup $\Gamma\triangleleft\mathrm{PSL}_2(\mathbb{Z})$ such that $\rho$ is a factor of $\rho_{\Gamma}$?

\begin{Remark}
Let $\{f_1, \ldots, f_g\}$ be a basis for $S_2(\Gamma)$, $\Gamma\triangleleft\mathrm{PSL}_2(\mathbb{Z})$ of finite-index and genus $g\geq 1$. Then $F := (f_1, \ldots, f_g)$ is a (non-zero) $\rho_{\Gamma}$-valued cusp form of weight two \cite{Candelori-Marks}. Since $\rho_{\Gamma}$ is unitarizable, this means that its {\em minimal weight} should be equal to two \cite{CandeloriFranc}.
\end{Remark}

\begin{Example}
Let $\Gamma(8)\triangleleft\mathrm{SL}_2(\mathbb{Z})$ be the principal congruence subgroup of level 8, and let $\Gamma$ be its image inside $\mathrm{PSL}_2(\mathbb{Z})$. The quotient $G = \mathrm{PSL}_2(\mathbb{Z})/\Gamma$ is a finite group of order 192, isomorphic to $\mathrm{SL}_2(\mathbb{Z}/8\mathbb{Z})$. Its character table is easily computable using GAP, and for each irreducible character we may compute its multiplicity inside $\rho_{\Gamma}$ using Theorem~\ref{theorem:Chevalley--WeilModForms}. We obtain the decomposition
\begin{gather*}
\rho_{\Gamma} \simeq \rho_1\oplus \rho_2,
\end{gather*}
where $\rho_1$ is irreducible of dimension two and $\rho_2$ is irreducible of dimension three. According to \cite[Proposition~2.5]{TubaWenzl}, such irreducible representations of $\mathrm{PSL}_2(\mathbb{Z})$ are entirely determined by their $\rho_i(T)$-eigenvalues. These can easily be computed again from the character table of~$G$. Using the formulas of \cite[Proposition~2.5]{TubaWenzl} we obtain the following model for $\rho_{\Gamma}$:
\begin{gather*}
\rho_{\Gamma}(T) = \left(\begin{matrix}
\zeta_8 & \zeta_8 & 0 & 0 & 0 \\
0 & \zeta_8^3 & 0 & 0 & 0 \\
0 & 0 & \zeta_8 & i-1 & i \\
0 & 0 & 0 & i & i \\
0 & 0 & 0 & 0 & \zeta^5_8
\end{matrix} \right), \qquad \rho_{\Gamma}(S) = \left(\begin{matrix}
0& \zeta^5_8 & 0 & 0 & 0 \\
\zeta_8^3 & 0 & 0 & 0 & 0 \\
0 & 0 & 0 & 0 & 1 \\
0 & 0 & 0 & -1 & 0 \\
0 & 0 & 1 & 0 & 0
\end{matrix} \right),
\end{gather*}
where $\zeta_8:= e^{2\pi i /8}$. Note in particular that $\rho_1$ is equivalent to the second representation in~\eqref{eqn:dim2IrrCanonical}.
\end{Example}

\section[The canonical representation of modular curves $X(\Gamma(p))$]{The canonical representation of modular curves $\boldsymbol{X(\Gamma(p))}$}\label{section:principalCongruenceSubgroups}

We now apply Theorem~\ref{theorem:Chevalley--WeilModForms} to a well-known family of normal subgroups of $\mathrm{PSL}_2(\mathbb{Z})$. Let $p$ be a prime and let $\Gamma(p)\triangleleft \mathrm{SL}_2(\mathbb{Z})$ be the principal congruence subgroup of level~$p$. For $p\geq 3$ we have and isomorphism $G \simeq \mathrm{PSL}_2(\mathbb{F}_p)$. The map $X(\Gamma(p)) \longrightarrow X(1)$ is a Galois cover of degree $|G| = \frac{(p-1)p(p+1)}{2}$, ramified above $\infty \in X(1)$. The local monodromy at $\infty \in X(\Gamma(p))$ is generated by $T$, whose image in $G$ has order $p$. Therefore the ramification degree is $e=p$.

By Theorem~\ref{theorem:Chevalley--WeilModForms} we can use the character table of $\mathrm{PSL}_2(\mathbb{F}_p)$ to decompose the representation
\begin{gather*}
\rho_{\Gamma(p)}\colon \ \mathrm{PSL}_2(\mathbb{Z}) \longrightarrow \mathrm{GL}(S_2(\Gamma(p)))
\end{gather*}
into irreducible representations. Note that $S_2(\Gamma(3))=0$, so assume $p\geq 5$. In this case the images of $R$ and $R^{-1}$ inside $G$ are conjugate, therefore $\operatorname{Tr}(\rho(R)) = \operatorname{Tr}(\rho(R^{-1}))$ for any representation $\rho$ of $G$. The formula of Theorem~\ref{theorem:Chevalley--WeilModForms} then simplifies to
\begin{gather}\label{equation:ChevalleyWeilCongruence}
d = -\frac{5}{12}\dim \rho + \sum_{k=1}^{p-1} N_k(\rho;T)\left(1 - \frac{k}{p}\right) - \frac{\operatorname{Tr}(\rho(S))}{4} - \frac{\operatorname{Tr}(\rho(R))}{3}.
\end{gather}

The character table of $\mathrm{PSL}_2(\mathbb{F}_p)$, $p\geq 5$ was originally computed by Frobenius and Schur (a~modern account can be found in \cite[Section~8]{GroupReps}) and it depends on whether $p\equiv 1,3 \,(4)$. We therefore break down the computation into these two cases. For ease of notation, we let $u_p$ be a generator for $\mathbb{F}_p^{\times}$ and we let $u_{p^2}$ be a~generator for~$\mathbb{F}_{p^2}^{\times}$. We also let $\varphi \in G$ be the linear transformation of~$\mathbb{F}_{p^2}$ sending~$x\mapsto u_{p^2}^{p-1}x$.

\subsection[Case $p\equiv 1 \, (4)$]{Case $\boldsymbol{p\equiv 1 \, (4)}$}

In this case
\begin{gather*}
S \sim \twobytwo{u_p^{\frac{p-1}{4}}}{0}{0}{u_p^{-\frac{p-1}{4}}}, \qquad R \sim \begin{cases} \twobytwo{u_p^{\frac{p-1}{3}}}{0}{0}{u_p^{-\frac{p-1}{3}}} &\text{if } p\equiv 1\, (3), \\
\\
\varphi^{\frac{p+1}{3}} &\text{if } p\equiv 2\, (3),
\end{cases}
\end{gather*}
and from the character table of \cite[Theorem~8.9]{GroupReps} we obtain the following table, displaying the traces of each irreducible character of $G$ evaluated at the matrices $I_2$, $S$, $R$, $T$ and $ \smalltwobytwo{1}{0}{u_p}{1}$:
\begin{gather}\label{table:p14}
\begin{array}{@{}c|c|c|c| c|c|c}
&\dim &S &R\;\text{ if } p\equiv 1\, (3) & R\;\text{ if } p\equiv 2\, (3) & T & \smalltwobytwo{1}{0}{u_p}{1} \\
\hline
\lambda &p&1&1&-1 & 0 & 0\\
\{\mu_s\}_{1\leq s \leq \frac{p-5}{4}} & p+1 & 2\cdot(-1)^s & \epsilon_3(s) & 0& 1 & 1 \\
\{\theta_t\}_{1\leq t \leq \frac{p-1}{4}} & p-1 & 0 & 0 & -\epsilon_3(t)& -1 & -1\\
\chi_1 & \frac{p+1}{2} & (-1)^{\frac{p-1}{4}} & 1 & 0 & \frac{1 + \sqrt{p}}{2} & \frac{1 - \sqrt{p}}{2}\\
\chi_2 &\frac{p+1}{2} & (-1)^{\frac{p-1}{4}} & 1 & 0&\frac{1 - \sqrt{p}}{2} & \frac{1 + \sqrt{p}}{2}\\
\end{array}
\end{gather}
where
\begin{gather*}
\epsilon_3(n) = \begin{cases} \hphantom{-}2 &\text{if } n\equiv 0\,(3), \\
 -1 &\text{if } n\equiv 1,2\,(3),
\end{cases}
\end{gather*}
and the irreducible characters $\lambda$, $\mu_s$, $\theta_t$, $\chi_1$, $\chi_2$ are labelled as in \cite[Theorem~8.9]{GroupReps}. The table contains all the data required to compute formula~\eqref{equation:ChevalleyWeilCongruence}, including the ramification data at~$\infty$, which is derived below in Proposition~\ref{prop:TRamification}. In order to state this computation, for any odd prime~$p$ and any integer~$a$ such that $\gcd(a,p)=1$ let
\begin{gather*}
\left(\frac{a}{p}\right)_L =
 \begin{cases} \hphantom{-}1 & \text{if $a$ is a square mod $p$}, \\
						 -1 & \text{if $a$ is a not a square mod $p$} \end{cases}
\end{gather*}
be the Legendre symbol.

\begin{Proposition}\label{prop:TRamification}For each one of the irreducible representations appearing in Table~{\rm \ref{table:p14}}, we have
\begin{gather*}
N_k(\lambda; T) = 1, \qquad N_k(\mu_s; T) = 1, \qquad N_k(\theta_t; T) = 1, \\
 N_k(\chi_1; T) = \begin{cases} 1 &\text{if } \left(\dfrac{k}{p}\right)_L =1, \vspace{1mm}\\
 0 &\text{if } \left(\dfrac{k}{p}\right)_L = -1,
\end{cases} \qquad N_k(\chi_2; T) = \begin{cases} 0 &\text{if } \left(\dfrac{k}{p}\right)_L =1, \vspace{1mm}\\
 1 &\text{if } \left(\dfrac{k}{p}\right)_L = -1,
\end{cases}
\end{gather*}
where $k=1,\ldots,p-1$.
\end{Proposition}

\begin{proof}For any representation $\rho$ of $G$, let $P(\rho;T)$ be the characteristic polynomial of $\rho(T)$. Since $\rho(T)^p = I_{\dim \rho}$, we know that the roots of $P(\rho;T)$ are necessarily $p$-th roots of unity. The polynomial itself can be computed from Table \ref{table:p14} by first computing the traces of $\rho(T^k)$, $k=0,\ldots, p-1$ and then compute from these traces the coefficients of $P(\rho;T)$ using Newton's identities. Now the subgroup $\big\{T^k\big\}$ breaks up into two conjugacy classes:
\begin{gather*}
\big\{T^k\big\} \cap [T] = \big\{ \smalltwobytwo{1}{r}{0}{1}\colon r \text{ is a square in } \mathbb{F}_p\big\}, \\
\big\{T^k\big\} \cap \big[\smalltwobytwo{1}{0}{u_p}{1}\big] = \big\{ \smalltwobytwo{1}{n}{0}{1}\colon n \text{ is not a square in } \mathbb{F}_p\big\},
\end{gather*}
therefore the coefficients of $P(\rho;T)$ are polynomials in the last two columns of Table \ref{table:p14}. For $\rho = \lambda, \mu_s$ and $\theta_t$, it follows that $P(\rho;T)$ has coefficients in $\mathbb{Q}$, and therefore
\begin{gather*}
\gcd\big(P(\rho;T), x^p-1\big) = \begin{cases} x-1 \\
 \Phi_p \\
 x^p-1 \end{cases} \in \mathbb{Q}[x],
\end{gather*}
where $\Phi_p$ is the $p$-th cyclotomic polynomial, which is irreducible over $\mathbb{Q}$. Suppose first $\rho=\lambda$. If $\gcd(P(\lambda;T),x^p-1) = (x-1)$ then necessarily $\lambda(T)= I_p$, which is impossible since Table \ref{table:p14} gives $\operatorname{Tr}\lambda(T) = 0 \neq p$. If $\gcd(P(\lambda;T),x^p-1) = \Phi_p$, then $P(\lambda;T)$ factors over $\mathbb{Q}$ as $\Phi_p \cdot(x - \zeta_p)$, for some $p$-th root of unity $\zeta_p$. But since $p\geq 5$, the only such polynomial $(x - \zeta_p)$ defined over~$\mathbb{Q}$ is~$(x-1)$, therefore $P(\lambda;T) = x^p-1$ necessarily, which gives the formula
\begin{gather*}
N_k(\lambda; T) = 1, \qquad k=1, \ldots, p-1.
\end{gather*}
For $\rho=\mu_s$, again we cannot have $\gcd(P(\mu_s;T),x^p-1) = (x-1)$ since none of the $\mu_s(T)$ is the identity, the trace being $1\neq 0$. If $\gcd(P(\mu_s;T),x^p-1) = \Phi_p$ then $P(\mu_s;T)$ factors as $\Phi_p\cdot q(x)$, for some quadratic polynomial $q(x)$. The only $\mathbb{Q}$-rational quadratic polynomials having only $p$-th roots as their solutions are $(x-1)^2$ and $\Phi_3$. Since $p\geq 5$, the latter cannot occur and $P(\mu_s;T) = \Phi_p\cdot(x-1)^2$. A similar reasoning shows that $P(\theta_t;T) = \Phi_p$, so the first row of Proposition~\ref{prop:TRamification} is proved. It remains to compute $P(\rho;T)$ for $\rho=\chi_1,\chi_2$. Note that for these two representations the last two columns of Table~\ref{table:p14} have coefficients in $\mathbb{Q}(\sqrt{p})$ and therefore $P(\rho;T)$ is $\mathbb{Q}(\sqrt{p})$-rational. Now the $p$-th cyclotomic polynomial factors over $\mathbb{Q}(\sqrt{p})$ as
\begin{gather*}\Phi_p = \Phi^{(r)}_p\cdot \Phi^{(n)}_p, \qquad \in \mathbb{Q}(\sqrt{p})[x],\end{gather*}
where $\Phi^{(r)}_p$ (resp.~$\Phi^{(n)}_p$) is the irreducible monic polynomial of degree $(p-1)/2$ over $\mathbb{Q}(\sqrt{p})$ whose roots are the $p$-th roots of unity of the form $e^{2\pi i r/p}$ (resp. $e^{2\pi i n/p}$), with $r$ (resp. $n$) a quadratic residue mod $p$ (resp. quadratic non-residue mod $p$). Let now $\rho=\chi_1$. Then $\gcd(P(\chi_1;T), x^p-1)$ is a non-constant $\mathbb{Q}(\sqrt{p})$-rational polynomial and it cannot be $(x-1)$ for otherwise $\chi(T)$ would be the identity, which is impossible since $\operatorname{Tr}(\chi_1(T)) \neq p$. Since $\deg P(\chi_1;T) = (p+1)/2$, the only other possibilities are that $\gcd(P(\chi_1;T), x^p-1) = \Phi^{(r)}_p$ or $\Phi^{(n)}_p $, so that $P(\chi_1;T) = (x-1)\Phi^{(r)}_p$ or $(x-1)\Phi^{(n)}_p$. To determine which one it is, note that the sum of the roots of $(x-1)\Phi^{(r)}_p$ can be computed by a Gauss sum and it is equal to $1/2(-1-\sqrt{p})$. This must coincide with the negative of $\operatorname{Tr}(\chi_1(T))$, so by Table~\ref{table:p14} we see that this is indeed the correct factorization. Similarly we may deduce that $P(\chi_2;T) = (x-1)\Phi^{(n)}_p$, which concludes the proof.
 \end{proof}

Putting everything into formula \eqref{equation:ChevalleyWeilCongruence} we get:

\begin{Theorem}\label{Thm:CWFormulaForP14} Let $p\geq 5$ be a prime, $p\equiv 1\,(4)$. Then the multiplicities of the irreducible representations of $\mathrm{PSL}_2(\mathbb{F}_p)$ inside the canonical representation $\rho_{\Gamma(p)}$ are given by
\begin{gather*}
d(\lambda) = \frac{p-9}{12} - \frac{1}{3}\left(\frac{p}{3}\right)_L, \\
d(\mu_s) = \frac{p-11}{12} -\frac{(-1)^s}{2} - \left(1+\left(\frac{p}{3}\right)_L\right) \frac{\epsilon_3(s)}{6}, \qquad 1\leq s \leq \frac{p-5}{4}, \\
d(\theta_t) = \frac{p-1}{12} + \left(1-\left(\frac{p}{3}\right)_L\right) \frac{\epsilon_3(t)}{6}, \qquad 1\leq t \leq \frac{p-1}{4}, \\
d(\chi_1) = \frac{p-11}{24} - \frac{(-1)^{\frac{p-1}{4}}}{4} - \frac{1}{6}\left(1+\left(\frac{p}{3}\right)_L\right), \\
d(\chi_2) = \frac{p-11}{24} - \frac{(-1)^{\frac{p-1}{4}}}{4} - \frac{1}{6}\left(1+\left(\frac{p}{3}\right)_L\right).
\end{gather*}
\end{Theorem}

\begin{proof}The formulas follow from a straightforward insertion of the data of Table~\ref{table:p14} into~\eqref{equation:ChevalleyWeilCongruence}. Note that for the multiplicities of $\chi_1$ and $\chi_2$ we have used the formulas
\begin{gather*}
\sum_{\substack{k=1\\ \left(\frac{k}{p}\right)_L =1}}^{p-1} \frac{k}{p} = \frac{p-1}{4} = \sum_{\substack{k=1\\ \left(\frac{k}{p}\right)_L = -1}}^{p-1} \frac{k}{p},
\end{gather*}
which are valid for $p\equiv 1\, (4)$.
\end{proof}

\subsection[Case $p\equiv 3 \, (4)$]{Case $\boldsymbol{p\equiv 3 \, (4)}$}

In this case
\begin{gather*}
S \sim \varphi^{\frac{p+1}{4}}, \qquad R \sim \begin{cases} \twobytwo{u_p^{\frac{p-1}{3}}}{0}{0}{u_p^{-\frac{p-1}{3}}} &\text{if } p\equiv 1\, (3), \\
\varphi^{\frac{p+1}{3}} &\text{ if } p\equiv 2\, (3),
\end{cases}
\end{gather*}
and from the character table of \cite[Theorem~8.11]{GroupReps} we obtain the following data (note that $T\sim \smalltwobytwo{1}{0}{u_p}{1}$ in this case)
\begin{gather*}
\begin{array}{@{}c|c|c|c| c | c | c}
&\dim &S &R\;\text{ if } p\equiv 1\, (3) & R\;\text{ if } p\equiv 2\, (3) & T & \smalltwobytwo{1}{0}{1}{1} \\
\hline
\lambda &p&-1&1&-1 & 0 & 0\\
\{\mu_s\}_{1\leq s \leq \frac{p-3}{4}} & p+1 & 0 & \epsilon_3(s) & 0& 1 & 1 \\
\{\theta_t\}_{1\leq t \leq \frac{p-3}{4}} & p-1 & -2\cdot(-1)^t & 0 & -\epsilon_3(t)& -1 & -1\\
\gamma_1 & \frac{p-1}{2} & -(-1)^{\frac{p+1}{4}} & 0 & -1 & \frac{-1 - \sqrt{-p}}{2} & \frac{-1 + \sqrt{-p}}{2}\\
\gamma_2 &\frac{p-1}{2} & -(-1)^{\frac{p+1}{4}} & 0 & -1& \frac{-1 + \sqrt{-p}}{2} & \frac{-1 - \sqrt{-p}}{2}
\end{array}
\end{gather*}
where the irreducible characters $\lambda$, $\mu_s$, $\theta_t$, $\chi_1$, $\chi_2$ are labelled as in \cite[Theorem~8.11]{GroupReps}. Exactly as in Proposition~\ref{prop:TRamification} we can derive from this table the ramification data at $\infty$:
\begin{gather*}
N_k(\lambda; T) = 1, \qquad N_k(\mu_s; T) = 1, \qquad N_k(\theta_t; T) = 1, \\
 N_k(\gamma_1; T) = \begin{cases} 1 &\text{if } \left(\dfrac{k}{p}\right)_L =-1, \vspace{1mm}\\
 0 &\text{if } \left(\dfrac{k}{p}\right)_L = 1,
\end{cases} \qquad N_k(\gamma_2; T) = \begin{cases} 0 &\text{if } \left(\dfrac{k}{p}\right)_L =-1, \vspace{1mm}\\
 1 &\text{if } \left(\dfrac{k}{p}\right)_L = 1,
\end{cases}
\end{gather*}
where $k=1,\ldots,p-1$. Putting everything into formula~\eqref{equation:ChevalleyWeilCongruence} we get:

\begin{Theorem}
\label{Thm:CWFormulaForP34}
Let $p\geq 5$ be a prime, $p\equiv 3\,(4)$. Then the multiplicities of the irreducible representations of $\mathrm{PSL}_2(\mathbb{F}_p)$ inside the canonical representation $\rho_{\Gamma(p)}$ are given by
\begin{gather*}
d(\lambda) = \frac{p-3}{12} - \frac{1}{3}\left(\frac{p}{3}\right)_L ,\\
d(\mu_s) = \frac{p-11}{12} - \left(1+\left(\frac{p}{3}\right)_L\right) \frac{\epsilon_3(s)}{6}, \qquad 1\leq s \leq \frac{p-3}{4}, \\
d(\theta_t) = \frac{p-1}{12} + \frac{(-1)^t}{2} + \left(1-\left(\frac{p}{3}\right)_L\right) \frac{\epsilon_3(t)}{6}, \qquad 1\leq t \leq \frac{p-3}{4}, \\
d(\gamma_1) = \frac{p-1}{24} - \frac{h(p)}{2} + \frac{(-1)^{\frac{p+1}{4}}}{4} + \frac{1}{6}\left(1-\left(\frac{p}{3}\right)_L\right), \\
d(\gamma_2) = \frac{p-1}{24} + \frac{h(p)}{2} + \frac{(-1)^{\frac{p+1}{4}}}{4} + \frac{1}{6}\left(1-\left(\frac{p}{3}\right)_L\right),
\end{gather*}
where $h(p)$ is the class number of $\mathbb{Q}(\sqrt{-p})$.
\end{Theorem}

\begin{proof}Only the formulas for $d(\gamma_1)$, $d(\gamma_2)$ require justification. Note that since $p\equiv 3\,(4)$
\begin{gather*}
\sum_{\substack{k=1\\ \left(\frac{k}{p}\right)_L = -1}}^{p-1} \frac{k}{p} - \sum_{\substack{k=1\\ \left(\frac{k}{p}\right)_L =1}}^{p-1} \frac{k}{p} = L\left (0,\left(\frac{\cdot}{p}\right)_L\right) \neq 0
\end{gather*}
and clearly $\sum\limits_{\substack{k=1\\ \left(\frac{k}{p}\right)_L = -1}}^{p-1} \frac{k}{p} + \sum\limits_{\substack{k=1\\ \left(\frac{k}{p}\right)_L = 1}}^{p-1} \frac{k}{p} = (p-1)/2$. Using the functional equation for the $L$-function and the analytic class number formula, we get
\begin{gather*}
\sum_{\substack{k=1\\ \left(\frac{k}{p}\right)_L =1}}^{p-1} \frac{k}{p} = \frac{p-1}{4} - \frac{h(p)}{2}, \qquad \sum_{\substack{k=1\\ \left(\frac{k}{p}\right)_L = -1}}^{p-1} \frac{k}{p} = \frac{p-1}{4} + \frac{h(p)}{2},
\end{gather*}
which gives the formulas for $d(\gamma_1)$, $d(\gamma_2)$.
\end{proof}

\begin{Remark}We thank Cameron Franc for showing us how to simplify $d(\gamma_1)$ and $d(\gamma_2)$ using Dirichlet's $L$-functions and the class number formula. Similar computations, but in a different context, have also appeared in \cite{CFrancKopp}.
\end{Remark}

\begin{Example}
The formulas of Theorems~\ref{Thm:CWFormulaForP14} and~\ref{Thm:CWFormulaForP34} show that the only prime~$p$ for which the canonical representation $\rho_{\Gamma(p)}$ is irreducible is the prime $p=7$. In this case the orbifold curve~$X(\Gamma(7))$ is an actual algebraic curve, and in fact it is a model for the {\em Klein quartic}, the unique Hurwitz surface of genus~3. Applying Theorem~\ref{Thm:CWFormulaForP34} we get that
$d(\gamma_2) = 1$, since $h(7)=1$, while all the other multiplicities are zero. Therefore $\rho_{\Gamma(p)} \simeq \gamma_2$. To write down explicit matrices for~$\rho_{\Gamma(p)}$, we know by \cite{TubaWenzl} that a three-dimensional irreducible representation of~$\mathrm{PSL}_2(\mathbb{Z})$ is entirely determined by the eigenvalues of the matrix~$\rho_{\Gamma(p)}(T)$. As in Theorem~\ref{Thm:CWFormulaForP34}, we know that these are $\zeta_7$, $\zeta_7^2$, $\zeta_7^4$, where $\zeta_7 := e^{2\pi i /7}$. According to \cite[Proposition~2.5]{TubaWenzl}, there is a basis where~$\rho_{\Gamma(p)}$ is given by the matrices
\begin{gather*}
\rho_{\Gamma(7)}(T) = \left(\begin{matrix}
\zeta_7 & \zeta^2_7 + \zeta^3_7 & \zeta_7^2 \\
0 & \zeta_7^2 & \zeta_7^2 \\
0 & 0 & \zeta_7^4
\end{matrix} \right), \qquad \rho_{\Gamma(7)}(S) = \left(\begin{matrix}
0 & 0 & 1 \\
0 & -1 & 0 \\
1 & 0 & 0
\end{matrix} \right).
\end{gather*}
\end{Example}

\section{The canonical representation of Fermat curves}\label{section:FermatCurves}

Let $\Gamma(2)\subseteq \mathrm{SL}_2(\mathbb{Z})$ be the principal congruence subgroup of level two. This is isomorphic to the free product on two generators, which can be chosen to be
\begin{gather*}
A = \twobytwo{1}{2}{0}{1}, \qquad B = \twobytwo{1}{0}{-2}{1}.
\end{gather*}
For any integer $N > 1$, let $\Phi(N)$ be the kernel of the composition
\begin{gather*}
\Gamma(2) \stackrel{\varphi^{\rm{ab}}}\longrightarrow \mathbb{Z}\times \mathbb{Z} \longrightarrow \mathbb{Z}/N\mathbb{Z} \times \mathbb{Z}/N\mathbb{Z}, \\
\smalltwobytwo{1}{2}{0}{1} \longmapsto (1,0), \qquad \smalltwobytwo{1}{0}{-2}{1} \longmapsto (0,1),
\end{gather*}
where the first map $\varphi^{\rm{ab}}$ is projection onto the abelianization and the second map is reduction modulo~$N$. The composition is clearly surjective, therefore $\Phi(N)$ is a normal subgroup of index~$N^2$ in $\Gamma(2)$. We also denote by $\Phi(N)$ the corresponding subgroup inside $\mathrm{PSL}_2(\mathbb{Z})$.

\begin{Lemma}\label{lemma:Phi(N)IsNormal}
For all $N>1$, $\Phi(N)$ is normal in $\mathrm{PSL}_2(\mathbb{Z})$.
\end{Lemma}

\begin{proof}By definition, $\Phi(N) = \varphi^{\rm{ab},-1}(N\mathbb{Z}\times N\mathbb{Z})$. For any $\gamma\in \mathrm{PSL}_2(\mathbb{Z})$, the conjugate subgroup $g^{-1} \Phi(N)g$ is the kernel of a surjective homomorphism $\Gamma(2)\rightarrow \mathbb{Z}/N\mathbb{Z}\times \mathbb{Z}/N\mathbb{Z}$. This map must factor through a map $\Gamma(2)^{\rm{ab}} = \mathbb{Z}\times \mathbb{Z}\rightarrow \mathbb{Z}/N\mathbb{Z}\times \mathbb{Z}/N\mathbb{Z}$, by the universal property of abelianizations. Any such map must have kernel equal to $N\mathbb{Z} \times N\mathbb{Z}$, therefore $g^{-1} \Phi(N)g = \varphi^{\rm{ab},-1}(N\mathbb{Z}\times N\mathbb{Z}) = \Phi(N)$.
\end{proof}

The normal subgroup $\Phi(N)$ can be described explicitly inside $\mathrm{PSL}_2(\mathbb{Z})$ as the normal closure
\begin{gather*}
\Phi(N) = \big\langle U^3, V^{2N} \big\rangle^{\mathrm{PSL}_2(\mathbb{Z})}, \qquad V := SR^{-1}, \qquad U := TV \in \mathrm{PSL}_2(\mathbb{Z}).
\end{gather*}

The quotient $G:= \mathrm{PSL}_2(\mathbb{Z})/\Phi(N)$ is the semi-direct product $(\mathbb{Z}/N\mathbb{Z})^2 \rtimes S_3$. The normal subgroup $(\mathbb{Z}/N\mathbb{Z})^2$ is generated by the images of $A$ and $B$ while the $S_3$-subgroup is generated by the images of $S$ and $U$. The action of $S_3$ on $(\mathbb{Z}/N\mathbb{Z})^2$ is given explicitly by
\begin{gather*}
SAS = B, \qquad UAU^2 = B, \qquad UBU^2 = A^{-1}B^{-1} \mod \Phi(N).
\end{gather*}

By Lemma \ref{lemma:Phi(N)IsNormal}, the curve $X(\Phi(N))$ is a Galois cover
\begin{gather*}
f(\Phi(N))\colon \ X(\Phi(N)) \longrightarrow X(1)
\end{gather*}
with Galois group isomorphic to $G$. Since $A = T^2$ and $A^N\in \Phi(N)$, the image of $T$ in $G$ has order $2N$, so the ramification level is $2N$. For all $N>1$, there is an embedding of~$X(\Phi(N))$ into~$\mathbb{P}^2$ cut out by the equation $X^N + Y^N = Z^N$ in homogeneous coordinates~\cite{Lang}. In other words, $X(\Phi(N))$ is a~uniformization of the {\em Fermat curve}~$F_N$ of exponent~$N$.

By construction, the group of automorphisms of the curve $X(\Phi(N))$ contains $G = (\mathbb{Z}/N\mathbb{Z})^2 \rtimes S_3$ and it is known that these are indeed all the automorphisms of the curve~\cite{Tzermias}. We now apply Theorem~\ref{theorem:Chevalley--WeilModForms} to compute the decomposition of the canonical representation
\begin{gather*}
\rho_{\Phi(N)}\colon \ G \longrightarrow \mathrm{GL}(S_2(\Phi(N)))
\end{gather*}
into irreducible representations. The genus of $X(\Phi(N))$ can easily be worked out to be $\frac{1}{2}(N-2)(N-1)$, so we may assume that $N>2$. Note that $\Phi(N)$ is a non-congruence subgroup for almost all $N$ \cite{Phillips-Sarnak}, therefore the decomposition of $\rho_{\Phi(N)}$ cannot be deduced from the calculations of Section~\ref{section:principalCongruenceSubgroups}, even for $N=p$ a prime.

All the irreducible representations of $G$ can be computed using the `little subgroups method' for semi-direct products by an abelian group \cite[Section~8.2]{Serre-LRFG}, \cite[Table~2]{Barraza-Rojas}, which can be applied as follows. Let $S_3$ act on $\operatorname{Hom}((\mathbb{Z}/N\mathbb{Z})^2, \mathbb{C}^{\times})$ by $g\chi(x) = \chi(g^{-1}xg)$ and let $\{\chi_i\}$ be representatives for the cosets under this action. Let $H_i\subseteq S_3$ be the stabilizer of $\chi_i$, so that $\chi_i$ can be extended to a character of $G_i = (\mathbb{Z}/N\mathbb{Z})^2\cdot H_i$. Let $\bar{\rho}$ be any representation of $H_i$, viewed as a~representation of $G_i$ under the quotient map $G_i \rightarrow H_i$ and let $\theta_{i,\bar{\rho}} = \mathrm{Ind}_{G_i}^G \bar{\rho}\otimes \chi_i$. It can be shown that $\theta_{i,\bar{\rho}}$ is irreducible and that all the irreducible representations of~$G$ are obtained in this way \cite[Proposition~25]{Serre-LRFG}. These calculations in our case break down into two cases, according to whether $3 \,|\, N$.

\subsection[Case $3\nmid N$]{Case $\boldsymbol{3\nmid N}$}

In this case the orbits of the action of $S_3$ on $\operatorname{Hom}((\mathbb{Z}/N\mathbb{Z})^2, \mathbb{C}^{\times})$ can only have order 1, 3 or 6 with stabilizers $H_i$ isomorphic to $S_3$, $C_2$ or trivial, respectively. Write $\chi_{\alpha, \beta}$ for the character in $\operatorname{Hom}((\mathbb{Z}/N\mathbb{Z})^2, \mathbb{C}^{\times})$ uniquely determined by
\begin{gather*}
\chi_{\alpha,\beta}(A) = \zeta_{N}^{\alpha}, \qquad \chi_{\alpha,\beta}(B) = \zeta_{N}^{\beta}, \qquad \zeta_N = e^{2 \pi i /N}.
\end{gather*}
Then the $S_3$-orbit of size one is $\{\chi_{0,0} = 1\}$ and the $S_3$-orbits of size three are of the form $\{ \chi_{\alpha, \alpha}, \chi_{\alpha, -2\alpha}, \chi_{-2\alpha, \alpha}\}$ for $\alpha = 1, \ldots, N-1$. There are a total of $3(N-1) + 1$ characters in these orbits. The remaining $N^2 - 3N + 2$ elements are partitioned into orbits of size 6, with trivial stabilizers. We let
\begin{gather*}
\{\chi_{\alpha_i, \beta_i} \},\qquad i = 1, \ldots, \frac{N^2 - 3N + 2}{6} = \frac{1}{6}(N-2)(N-1)
\end{gather*}
be a set of representatives of these orbits. It is easy to see that $\alpha_i$, $\beta_i$ can be chosen so that $\alpha_i, \beta_i \neq 0$. Assuming so in what follows slightly simplifies the formulas.

Now the orbit $\{\chi_{0,0} = 1\}$ produces via the small subgroup method 3 representations, corresponding to the 3 irreducible representations of $S_3$:
\begin{itemize}\itemsep=0pt
\item[(i)] the trivial representation $\rho_1$,
\item[(ii)] 1 one-dimensional irreducible representation $\rho_2$,
\item[(iii)] 1 two-dimensional irreducible representation $\rho_3$.
\end{itemize}

The orbits of cyclic stabilizer $C_2$, with representatives $\chi_{\alpha, \alpha}$ produce
\begin{itemize}\itemsep=0pt
\item[(iv)] $(N-1)$irreducible 3-dimensional representations $\rho^+_{\alpha}$, $\alpha \in \{1,\ldots, N-1\}$,
\item[(v)] $(N-1)$ irreducible 3-dimensional representations $\rho^{-}_{\alpha}$, $\alpha \in \{1,\ldots, N-1\}$.
\end{itemize}
Each class comes from inducing $\bar{\rho}^{\pm} \otimes \chi_{\alpha}$ to $G$, where $\rho^{+}$ is the trivial representation of~$C_2$ and~$\rho^{-}$ is the non-trivial. Finally, the orbits of representative $\chi_{\alpha_i, \beta_i}$ give
\begin{itemize}\itemsep=0pt
\item[(vi)] $\frac{1}{6}(N-2)(N-1)$ irreducible 6-dimensional representations $\rho_{\alpha_i,\beta_i}$,
\end{itemize}
each representation being given by inducing $\chi_{\alpha_i, \beta_i}$ to all of $G$. Explicit matrices for the genera\-tors~$S$, $U$, $A$ can easily be written down as in~\cite{Barraza-Rojas}, from which we also borrowed the notation (a slight correction is needed in Table~2 of~\cite{Barraza-Rojas} for the matrices $\rho_{\alpha}^{\pm}(A)$ and $\rho_{\alpha,\beta}(A)$). Using the relation
\begin{gather*}
R = A^{-1}U \mod \Phi(N)
\end{gather*}
the traces of $\rho(S)$, $\rho(R)$, $\rho\big(R^{-1}\big)$ and the eigenvalues of the $\rho(T)$-matrix can then be computed from the explicit formulas and they are given in table 
\begin{gather*}
\begin{array}{@{}c|c|c|c| c | c }
&\dim &S &R & R^{-1} & T-\text{eigenvalues} \\
\hline
\rho_2 &1&-1&1&1 & -1 \\
\rho_3 & 2 & 0 & -1 & -1& 1,-1 \\
\rho^+_{\alpha} & 3 & 1 & 0 & 0 & \zeta_N^{-\alpha}, \zeta^{\alpha}_{2N},-\zeta^{\alpha}_{2N} \\
\rho^-_{\alpha} & 3 & -1 & 0 & 0 & -\zeta_N^{-\alpha}, -\zeta^{\alpha}_{2N},\zeta^{\alpha}_{2N} \\
\rho_{\alpha_i,\beta_i} &6 & 0 & 0 & 0 & \pm \zeta_{2N}^{\alpha_i}, \pm\zeta_{2N}^{\beta_i}, \pm\zeta_{2N}^{-(\alpha_i + \beta_i)}
\end{array}
\end{gather*}
where $\zeta_N = e^{2\pi i/N}$ and $\zeta_{2N} = e^{2\pi i/2N}$. Applying Theorem~\ref{theorem:Chevalley--WeilModForms} we obtain the following.

\begin{Theorem}\label{thm:CWFermatCurveNDoesNotDivide3}Let $N>2$ be an integer such that $3\nmid N$. Then the multiplicities of the irreducible representations of $G = (\mathbb{Z}/N\mathbb{Z})^2 \rtimes S_3$ inside the canonical representation $\rho_{\Phi(N)}$ are given by
\begin{gather*}
d(\rho_1) = d(\rho_2) = d(\rho_3) = 0, \qquad
 d(\rho_{\alpha}^+) = 0, \qquad \alpha=1, \ldots, N-1, \\
d(\rho_{\alpha}^-) = \begin{cases} 1 & \text{if } \alpha = 1, \ldots, \left\lfloor \dfrac{N-1}{2} \right\rfloor, \vspace{1mm}\\
 0 & \text{if } \alpha = \left\lfloor \dfrac{N-1}{2} \right\rfloor + 1, \ldots, N - 1,\end{cases} \\
d(\rho_{\alpha_i,\beta_i}) = \begin{cases} 1 & \text{if } \alpha_i + \beta_i < N, \\
 0 & \text{if } \alpha_i + \beta_i \geq N, \end{cases} \qquad \alpha_i, \beta_i \neq 0 .
\end{gather*}
\end{Theorem}

\begin{Example}Let $N=7$. There are 5 $S_3$-orbits of size 6 in $\operatorname{Hom}((\mathbb{Z}/7\mathbb{Z})^2, \mathbb{C}^{\times})$, and the representatives for these orbits can be chosen to be
$\{\chi_{6,1}, \chi_{5,2}, \chi_{4,3}, \chi_{2,1}, \chi_{6,3}\}$. By Theorem~\ref{thm:CWFermatCurveNDoesNotDivide3} we deduce that
\begin{gather*}
\rho_{\Phi(7)} \sim \rho^-_1 \oplus \rho^-_2 \oplus \rho^-_3 \oplus \rho_{2,1}.
\end{gather*}
An explicit model for $\rho_{\Phi(7)}$ can then be written down using the following formulas
\begin{gather*}
\rho_{\alpha}^-(T) = \left(\begin{matrix}
0 & 0 & \zeta_7^{-\alpha} \\
0 & \zeta_7^{-\alpha} & 0 \\
-\zeta_7^{2\alpha} & 0 & 0
\end{matrix}\right) , \qquad \rho_{\alpha}^-(S)= \left(\begin{matrix}
-1 & 0 & 0 \\
0 & 0 & -1 \\
0 & -1 & 0
\end{matrix}\right) , \qquad \alpha=1,2,3,\\
\rho_{2,1}(T) =\left(\begin{matrix}
0 & 0 & 0 & 0 & \zeta_7^{-1} & 0 \\
0 & 0 & 0 & 0 & 0 & \zeta_7^{-2} \\
0 & 0 & 0 & \zeta_7^{3} & 0 & 0 \\
0 & 0 & \zeta_7^{-1} & 0 & 0 & 0 \\
\zeta_7^{3} & 0 & 0 & 0 & 0 & 0 \\
0 & \zeta_7^{-2} & 0 & 0 & 0 & 0
\end{matrix} \right) ,\qquad \rho_{2,1}(S)= \left(\begin{matrix}
0 & 0 & 0 & 1 & 0 & 0 \\
0 & 0 & 0 & 0 & 1 & 0 \\
0 & 0 & 0 & 0 & 0 & 1 \\
1 & 0 & 0 & 0 & 0 & 0 \\
0 & 1 & 0 & 0 & 0 & 0 \\
0 & 0 & 1 & 0 & 0 & 0
\end{matrix} \right).
\end{gather*}
\end{Example}

\subsection[Case $3\,|\, N$]{Case $\boldsymbol{3\,|\, N}$} In this case there are 3 characters with full stabilizer $S_3$, given by $\{\chi_{0,0}\}$, $\{\chi_{N/3,N/3}\}$ and $\{\chi_{2N/3,2N/3}\}$ (and $N-3$ remaining orbits with $C_2$-stabilizers). The two additional characters each produce
\begin{itemize}\itemsep=0pt
\item[(vii)] 2 one-dimensional (non-trivial) representations $\rho^1_{N/3}$ and $\rho^1_{2N/3}$,
\item[(viii)] 2 one-dimensional (non-trivial) representations $\rho^2_{N/3}$ and $\rho^2_{2N/3}$,
\item[(ix)] 2 two-dimensional representations $\rho^3_{N/3}$ and $\rho^3_{2N/3}$,
\end{itemize}
the first corresponding to the trivial representation of $S_3$, the second to the sign representation, and the third to the 2-dimensional representation of $S_3$. For each one of these new representations we can compute the following data:
\begin{gather*}
\begin{array}{@{}c|c|c|c| c | c }
&\dim &S &R & R^{-1} & T\text{-eigenvalues} \\
\hline
\rho^1_{N/3} &1&1&\zeta_3^{-1} & \zeta_3 & \zeta_3^{-1} \\
\rho^1_{2N/3} &1&1&\zeta_3 & \zeta_3^{-1} & \zeta_3 \\
\rho^2_{N/3} &1&-1&\zeta_3^{-1} & \zeta_3 & -\zeta_3^{-1} \\
\rho^2_{2N/3} & 1 & -1 & \zeta_3 & \zeta_3^{-1} & -\zeta_3 \\
\rho^3_{N/3} & 2 & 0 & -\zeta_3^{-1} & -\zeta_3 & \zeta_6, -\zeta_6 \\
\rho^3_{2N/3} & 2 & 0 & -\zeta_3 & -\zeta_3^{-1} & -\zeta_3, \zeta_3 \\
\end{array}
\end{gather*}

Applying Theorem~\ref{theorem:Chevalley--WeilModForms} we get

\begin{Theorem}\label{thm:CWFermatCurve3DividesN} Let $N>2$ be an integer such that $3\,|\, N$. Then the multiplicities of the irreducible representations of $G = (\mathbb{Z}/N\mathbb{Z})^2 \rtimes S_3$ inside the canonical representation $\rho_{\Phi(N)}$ are given by
\begin{gather*}
d(\rho_1) = d(\rho_2) = d(\rho_3) = 0, \\
d(\rho_{\alpha}^+) = 0, \qquad \alpha = 1, \ldots, N-1, \qquad \alpha\neq N/3, 2N/3, \\
d(\rho_{\alpha}^-) = \begin{cases} 1 & \text{if } \alpha = 1, \ldots, \left\lfloor \dfrac{N-1}{2} \right\rfloor, \vspace{1mm}\\
 0 & \text{if } \alpha = \left\lfloor \dfrac{N-1}{2} \right\rfloor + 1, \ldots, N - 1, \end{cases} \qquad \alpha\neq N/3, 2N/3, \\
d(\rho_{\alpha_i,\beta_i}) = \begin{cases} 1 & \text{if } \alpha_i + \beta_i < N, \\
 0 & \text{if } \alpha_i + \beta_i \geq N, \end{cases} \qquad \alpha_i, \beta_i \neq 0 , \\
 d\big(\rho^1_{N/3}\big) = d\big(\rho^1_{2N/3}\big) = d\big(\rho^2_{2N/3}\big) = 0, \qquad
 d\big(\rho^2_{N/3}\big) = 1, \qquad
 d\big(\rho^3_{N/3}\big) = d\big(\rho^3_{2N/3}\big) = 0.
\end{gather*}
\end{Theorem}

\begin{Example} Let $N=6$. There are 4 $S_3$-orbits of size 6 in $\operatorname{Hom}\big((\mathbb{Z}/6\mathbb{Z})^2, \mathbb{C}^{\times}\big)$, and representatives for these orbits can be chosen to be $\{ \chi_{5,1}, \chi_{4,2}, \chi_{5,4}, \chi_{3,2}\}$. By Theorem~\ref{thm:CWFermatCurve3DividesN} we deduce that
\begin{gather*}
\rho_{\Phi(6)} \sim \rho_2^2 \oplus \rho_1^- \oplus \rho_{3,2}
\end{gather*}
and an explicit model for $\rho_{\Phi(6)}$ can be written down using
\begin{gather*}
\rho_2^2(T) = e^{2\pi i/6} , \qquad \rho_2^2(S) = -1, \\
\rho_{\alpha}^-(T) = \left(\begin{matrix}
0 & 0 & \zeta_6^{-1} \\
0 & \zeta_6^{-1} & 0 \\
-\zeta_3 & 0 & 0
\end{matrix}\right) , \qquad \rho_{\alpha}^-(S)= \left(\begin{matrix}
-1 & 0 & 0 \\
0 & 0 & -1 \\
0 & -1 & 0
\end{matrix}\right), \\
\rho_{3,2}(T) =\left(\begin{matrix}
0 & 0 & 0 & 0 & \zeta_6^{-2} & 0 \\
0 & 0 & 0 & 0 & 0 & \zeta_6^{-3} \\
0 & 0 & 0 & \zeta_6^{5} & 0 & 0 \\
0 & 0 & \zeta_6^{-2} & 0 & 0 & 0 \\
\zeta_6^{5} & 0 & 0 & 0 & 0 & 0 \\
0 & \zeta_6^{-3} & 0 & 0 & 0 & 0
\end{matrix} \right) ,\qquad \rho_{3,2}(S)= \left(\begin{matrix}
0 & 0 & 0 & 1 & 0 & 0 \\
0 & 0 & 0 & 0 & 1 & 0 \\
0 & 0 & 0 & 0 & 0 & 1 \\
1 & 0 & 0 & 0 & 0 & 0 \\
0 & 1 & 0 & 0 & 0 & 0 \\
0 & 0 & 1 & 0 & 0 & 0
\end{matrix} \right).
\end{gather*}
\end{Example}

\pdfbookmark[1]{References}{ref}
\LastPageEnding

\end{document}